\documentclass[preprint,11pt]{amsart}%
\usepackage[margin=1.2in]{geometry}
\usepackage{amsmath}
\usepackage{amssymb}
\usepackage{amsfonts}
\usepackage{upref,amsthm,amsxtra,exscale}
\usepackage{cite}
\usepackage[colorlinks=true,urlcolor=blue,
citecolor=red,linkcolor=blue,linktocpage,pdfpagelabels,
bookmarksnumbered,bookmarksopen]{hyperref}
\usepackage{epsfig,graphics,color}
\usepackage{graphicx}%
\newtheorem{theorem}{Theorem}[section]
\theoremstyle{plain}

\newtheorem{lemma}[theorem]{Lemma}

\newtheorem{remark}{Remark}
\numberwithin{equation}{section}

\begin{document}
\title[Nonlinear
Schr\"{o}dinger Equations]{A Note on Nonlinear
	Schr\"{o}dinger Equations: Unveiling the Relation Between Spectral Gaps and the Nonlinearity Behavior}
\author{Mayra Soares}
\address{Departamento de Matem\'{a}tica, UNB, 70910-900 Bras\'{\i}lia, Brazil.}
\email{ssc\_mayra@hotmail.com}
\author{Liliane A. Maia}
\address{Departamento de Matem\'{a}tica, UNB, 70910-900 Bras\'{\i}lia, Brazil.}
\email{lilimaia@unb.br}
\thanks{FAPDF 	00193.00001738/2018-75
and CNPq/PQ 308173/2014-7}
\date{\today}

\begin{abstract}
	We establish the existence of a nontrivial weak solution to strongly indefinite asymptotically linear and superlinear Schr\"odinger equations. The novelty is to identify the essential relation between the spectrum of the operator and the behavior of the nonlinear term, in order to weaken the necessary assumptions to obtain a linking structure to the problem, for instance to allow zero  being in the spectrum or the nonlinearity being sign-changing. Our main difficulty is to overcome the lack of monotonicity on the nonlinear term, as well, as the lack of compactness since the domain is unbounded. With this purpose, we require periodicity on $V$.

%\textsc{MSC2010: }35Q55 (35B09, 35J20).\medskip\medskip\medskip

\end{abstract}
\maketitle

\section{Introduction}

\label{sec:introduction}

\qquad In this paper we study the existence of a nontrivial weak solution to the following problem
\begin{equation}
\left\{
\begin{array}
[c]{l}%
-\Delta u+V(x)u=f(u),\\
u\in{H}^{1}(\mathbb{R}^{N}),
\end{array}
\right.  \tag{$P_V$}\label{prob}%
\end{equation}
for $N\geq3$, where $V$ is $1$-periodic
% or asymptotically $1$-periodic
and $f$ is an asymptotically linear or superlinear nonlinearity. Our goal is to shed light on the role played by the interaction between gaps of the spectrum of the Schr\"odinger operator and the asymptotic behavior of $f(s)/s$ as $s$ goes to zero or to infinity in other to loosen the hypotheses usually required in the literature. Problem (\ref{prob}) arises in the study of stationary states of nonlinear
Schr\"{o}dinger or Klein-Gordon type equations and has been studied
extensively in order to reach the most general assumptions which enable to solve it.
 
\qquad We assume the following hypotheses of behavior on the nonlinearity:
\medskip\newline
$(f_1)\quad
f\in\mathcal{C}(\mathbb{R})$, $f(0)=0$, there exists $\displaystyle\lim_{s \to 0}\dfrac{f(s)}{s}=f'(0)$ and for $F(s):=\displaystyle\int_{0}^{s}f(t)\,dt$
\[
\dfrac{2F(s)}{s^2} \geq f'(0) \quad \text{for \ all\ } s\in \mathbb{R};
\]
$(f_2)$\quad There exist $a>f'(0)$ and $p\in [2,2^*)$ such that 
\[
\displaystyle\lim_{|s| \to +\infty}\dfrac{f(s)}{s}\geq a \quad \text{and} \quad \displaystyle\limsup_{|s|\to+\infty}\dfrac{f(s)}{s^{p-1}} < +\infty; 
\]

\begin{remark}
For the asymptotically linear case we have $f(s)/s \to a$ as $|s|\to +\infty$, hence for all $p\in [2,2^*)$, it satisfies the second limit in $(f_2)$. Thereby, for the asymptotically linear case hypothesis $(f_2)$ can be rewritten as
\medskip
\newline
$(f_2)'$\quad There exists $a>f'(0)$  such that
\[
\displaystyle\lim_{|s| \to +\infty}\dfrac{f(s)}{s}= a.
\]
Whereas, for the superlinear case, $f(s)/s \to +\infty$ as $|s|\to +\infty$, hence for all $a>f'(0)$ the first limit in $(f_2)$ is satisfied. Nevertheless, is necessary to require some $p \in (2,2^*)$ satisfying the second limit in $(f_2)$ to ensure $f$ has sub-critical growth. Thereby, for the superlinear case hypothesis $(f_2)$ can be rewritten as
\medskip
\newline
$(f_2)''$\quad There exists $p\in (2,2^*)$ such that 
\[
\displaystyle\lim_{|s| \to +\infty}\dfrac{f(s)}{s}=+\infty \quad \text{and} \quad \displaystyle\limsup_{|s|\to+\infty}\dfrac{f(s)}{s^{p-1}} < +\infty.
\]
Moreover, we would like to point out that throughout the arguments presented here, it is only essential to distinguish $(f_2)'$ from $(f_2)''$ when we are arguing about the boundedness of Cerami sequences for the functional associated to problem (\ref{prob}).
\end{remark}

\qquad Defining  
operator $A:= -\Delta + V(x)$,
as an operator of $L^2(\mathbb{R}^N)$ and denoting by $\sigma(A)$ the spectrum of $A$ and by $\sigma_{ess}(A)$ the essential spectrum of $A$, we assume the following spectral hypotheses on $V$: 
\medskip
\newline
$(V_1)$\quad There exists $(\sigma^-,\sigma^+)\subset \mathbb{R}\setminus\sigma(A)$, a spectral gap of $\sigma(A)$ such that ${\sigma^- < f'(0) < \sigma^+<a}$.
\medskip
\newline
$(V_2)$\quad $V\in L^\infty(\mathbb{R}^N)$ is a non constant $1$-periodic
% or asymptotically $1$-periodic
function in $x_i$ for $i = 1, . . . ,N.$
\medskip
\begin{remark}\label{r2}
 It is important to highlight that the relation between $\sigma(A)$ and $f$ is established by the limits of $\dfrac{f(s)}{s}$ at the origin and at infinity, which must be in some spectral gap and after the same spectral gap, respectively. For instance, it is straightforward to verify that the asymptotically linear
model nonlinearity $f(s)=\frac{as^{3} -s }{1+s^{2}}$ satisfies hypotheses $(f_1)-(f_2)$, with $f'(0)=-1$. In this case, assumption $(V_1)$ requires $-1\notin \sigma(A)$, but it is not required $0\notin \sigma(A)$ as usual. In addition, note that $(V_1)$ requires $a > \sigma^+$, but not $a>0$ as usual, since $\sigma^+$ and even $a$ can be negative. Furthermore, for the superlinear case, the constant $a$ can be any real number satisfying $\sigma^+ <a$, since the first limit in $(f_2)$ is infinity.
\end{remark}
\begin{remark}If $V \in L^\infty(\mathbb{R}^N)$ is periodic, then $A$ is a self-adjoint operator with purely absolutely continuous spectrum $\sigma(A) = \sigma_{ess}(A)= \displaystyle\cup[a_i,b_i]$, hence $f'(0)\notin \sigma(A)$ means that $f'(0) \in (b_i,a_{i+1})$ for some $i$, a spectral gap of $A$, then we denote this gap by $(\sigma^-,\sigma^+)$ and hypothesis $(V_1)$ makes sense in view of hypothesis $(V_2)$.
	%Furthermore, if $V$ is asymptotically periodic, $\sigma(A)$ may not be purely absolutely continuous, nevertheless, it keeps having gaps, then it is sufficient to ask for $f'(0)\notin \sigma(A)$ and lying on some gap of $\sigma(A)$ as in $(V_1)$.
\end{remark}

\qquad To the best of our knowledge, these assumptions on problem (\ref{prob}) generalize usual conditions in the literature for this type of problem. Indeed, given a periodic $V$,  we only require on $f$ a necessary relation with some gap in $\sigma(A)$. This type of assumption allows $f$ being even negative in some cases.

\qquad Setting $Q: \mathbb{R}\to\mathbb{R}$, given by
\[
Q(s) := f(s)s - 2F(s),
\]
our main result is stated as follows.
\medskip
\begin{theorem}\label{main}
	Assume that $(V_1)-(V_2)$ and $(f_1)-(f_2)$ hold true. If $(f_2)'$ holds with
	\medskip
	\newline
	$(f_3)$\quad $Q(s)\geq0$ and
	\[\lim_{|s| \to +\infty}Q(s)= +\infty;
	\]
	or if $(f_2)''$ holds with
	\medskip
	\newline
	$(f_4)$\quad There exist $c_0>0$ and $\theta>\min\{1,\dfrac{N}{2}(p-2)\}$ such that 
	\[Q(s)\geq c_0|s|^\theta, \; \forall \; s\in \mathbb{R};
	\]
	then problem \emph{(\ref{prob})} has a nontrivial weak solution.
\end{theorem}

\begin{remark}
	Hypotheses $(f_3)$ and $(f_4)$ are only required in order to guarantee the boundedness of Cerami sequences for $I$ in asymptotically linear and superlinear cases, respectively.  
\end{remark}

\medskip

\qquad Periodic and asymptotically periodic nonlinear Schr\"odinger equations have been extensively studied in the last 20 years, hence there exists a vast bibliography on the subject. However, under the purpose of focusing on our contribution we would like to mention those works more closely related with ours. Once we have been studying indefinite problems in $\mathbb{R}^N$, we were first motivated by G. Li and A. Szulkin in \cite{LS}, where they prove an abstract linking theorem involving a weaker topology. Their approach allows to obtain a Cerami sequence for the associated functional without assuming any compactness condition. On the other hand, they strongly rely on a auxiliary problem, which requires a monotonicity condition on the nonlinearity, in order to get a nontrivial critical point. 

\qquad Here, we are able to apply the abstract result by G. Li and A. Szulkin in \cite{LS}, but to get a nontrivial critical point, we exploit either periodicity
% or asymptotic periodic profile
and spectral properties, avoiding any monotonicity assumption. Similar ideas were used by Ding Y. and C. Lee in \cite{DL}, but under different hypotheses on the nonlinear term, other related works can be found in \cite{D}. Nevertheless, all these papers require zero lying on a gap of the spectrum, which is pretty common in the literature, we refer to \cite{CLX,MS2, M, Pa} and references therein. In this sense, our improvement is to realize this assumption is not necessary, generalizing for many problems. In fact, it is only necessary to exist a gap in the spectrum and the key for solving this type of problem is to work with nonlinearities which are related to the mentioned spectral gap in a suitable way, see hypothesis $(V_1)$. Furthermore, to the best of our knowledge, all works in this scenario assume either $0\notin \sigma(A)$ or even zero as a boundary point in a gap of $\sigma(A)$, see for instance \cite{BD}, but never zero as a interior point of $\sigma(A)$. However, we understand that zero is not the essence of this issue, since we can treat problems with zero as a interior point of $\sigma(A)$, by working with a translated problem.

\qquad Indeed, appropriating the ideas presented by L. Jeanjean in \cite{J}, we observe that ${u \in H^1(\mathbb{R}^N)}$ is a solution to (\ref{prob}) if and only if $u$ is a solution to
\begin{equation}
\left\{
\begin{array}
[c]{l}%
-\Delta u+V_0(x)u=f_0(u),\\
u\in{H}^{1}(\mathbb{R}^{N}),
\end{array}
\right.  \tag{$P_{V_0}$}\label{prob_0}%
\end{equation}
where 
\begin{eqnarray}\label{def}
\quad \quad \quad \quad V_0(x) &:=& V(x)-f'(0)\nonumber\\
\quad \quad \quad \quad f_0(s) &:=& f(s) - f'(0)s, \quad \text{for \ all\; } s \in \mathbb{R}.
\end{eqnarray}
\medskip
\newline
Therefore, we are going to find a nontrivial weak solution to problem (\ref{prob_0}) instead of (\ref{prob}). In view of $(f_1)-(f_2)$ and $(V_1)-(V_2)$ we obtain theirs correspondent versions for problem (\ref{prob_0}), namely, $(f_1)_0-(f_2)_0$ and $(V_1)_0-(V_2)_0$ on $f_0$ and $V_0$, respectively, which are more easygoing to treat than $(f_1)-(f_2)$ and $(V_1)-(V_2)$ on $f$ and $V$. They are stated as follows:
\medskip\newline
$(f_1)_0\quad
f_0\in\mathcal{C}(\mathbb{R})$, $f_0(0)=0$, there exists $\displaystyle\lim_{|s| \to 0}\dfrac{f_0(s)}{s}= 0$ and for $F_0(s):=\displaystyle\int_{0}^{s}f_0(t)\,dt$
\[
{F_0(s)} \geq 0 \quad \text{for \ all\ } s\in \mathbb{R};
\]
$(f_2)_0$\quad For $a_0:= a - f'(0)$ and for some $p\in[2,2^*)$ it holds
\[
\displaystyle\lim_{|s| \to +\infty}\dfrac{f_0(s)}{s}\geq a_0\quad \text{and} \quad \limsup_{|s|\to+\infty}\dfrac{f_0(s)}{s^{p-1}} <+\infty.
\]

\qquad Defining $A_0 := A - f'(0) =  -\Delta + V_0(x)$ and denoting the spectrum of $A_0$ by $\sigma(A_0)$,
\bigskip
\newline
$(V_1)_0$\quad There exists $\big(\sigma^- - f'(0), \sigma^+-f'(0)\big)\subset \mathbb{R}\setminus \sigma(A_0)$ a spectral gap of $\sigma(A_0)$ such that
\[
a_0= a - f'(0) > \sigma^+ - f'(0)>0>\sigma^--f'(0);
\]
$(V_2)_0$\quad $V_0\in L^\infty(\mathbb{R}^N)$ is a non constant $1$-periodic
% or asymptotically $1$-periodic
function in $x_i$ for $i=1,...,N$. 
\bigskip

\qquad Moreover, setting $Q_0(s) := f_0(s)s - 2F_0(s) = Q(s)$ for $s\in \mathbb{R}$, we observe that $(f_3)$ and $(f_4)$ are the same on $f$ or $f_0$.
\medskip

\qquad Theorem \ref{main} is proved by applying variational methods. Let us briefly
highlight some technical details. Setting the Hilbert space
%
%$$E := \{u \in H^1(\mathbb{R}^N): \int_{\mathbb{R}^N}V(x)u^2(x)\;dx <+\infty\},$$
%
$E:=\Big(H^1_0(\mathbb{R}^N), ||\cdot||\Big)$, where $||\cdot||$ is the norm induced by operator $A_0$ and considering $\{\mathcal{E}(\lambda): \lambda \in \mathbb{R}\}$ as the spectral family of operator $A_0$, we set $E^+ \subset E$ as the subspace given by $E^+ := \Big(I - \mathcal{E}(0)\Big)E$ where $A_0$ is positive definite and ${E^- := \mathcal{E}(0)E}$, the subspace where $A_0$ is negative definite, hence ${E = E^+ \oplus E^-}$. In view of $(V_1)_0$ we have that ${\big(\sigma^- - f'(0), \sigma^+-f'(0)\big)}$ is a spectral gap of $A_0$, with ${\sigma^- - f'(0)<0< \sigma^+-f'(0)}$, then by the spectral family definition, one has
\begin{equation}\label{e00}
\int_{\mathbb{R}^N}\Big(|\nabla u^+(x)|^2+ V_0(x)(u^+(x))^2\Big)\;dx \geq \Big(\sigma^+- f'(0)\Big)\int_{\mathbb{R}^N}(u^+(x))^2\;dx,
\end{equation}
for all $u^+ \in E^+$ and
\begin{equation}\label{e000}
-\int_{\mathbb{R}^N}\Big(|\nabla u^-(x)|^2+ V_0(x)(u^-(x))^2\Big)\;dx \geq \Big(f'(0) - \sigma^-\Big) \int_{\mathbb{R}^N}(u^-(x))^2\;dx,
\end{equation}
for all $u^- \in E^-.$  
Hence the following inner product is well defined and we are focused on looking for a nontrivial weak solution to (\ref{prob_0}) on the Hilbert space $E$ endowed with this suitable inner product

\begin{equation}\label{IP}
\big( u,v \big) = \left\{
\begin{array}{lllll}
\quad \displaystyle\int_{\mathbb{R}^N}\Big(\nabla u(x) \nabla v(x) + V_0(x)u(x)v(x)\Big)dx = (A_0u,v)_{L^2(\mathbb{R}^N)}  \quad \quad \ \quad   \quad \text{if} \ \ u, v \in E^+, \\
\\
\quad -\displaystyle\int_{\mathbb{R}^N}\Big(\nabla u(x) \nabla v(x) + V_0(x)u(x)v(x)\Big)dx = -(A_0u,v)_{L^2(\mathbb{R}^N)} \  \quad \  \quad\text{if} \ \ u,v \in E^-,\\
\\
%\quad \displaystyle\int_{\mathbb{R}^N}u(x)v(x)\;dx = (u,v)_{L^2(\mathbb{R}^N)} \quad \quad \quad  \quad \quad \quad \quad \quad \quad \quad \quad \quad \quad  \quad \quad \quad \quad \quad  \text{if} \ \ u,v \in E^0,\\
%\\
\quad 0  \quad \quad  \quad \quad \quad \quad \quad  \quad   \quad \quad \quad \quad \quad \quad \quad \quad \quad \quad \quad  \quad \quad \quad \quad \; \text{if} \ \ u \in E^i,\; \ v \in E^j, i\not=j,
\end{array}
\right.
\end{equation}
and the corresponding norm $||u||^2 := (u,u)$ for all $u \in E$, which is equivalent to the standard norm in $H^1(\mathbb{R}^N)$.

\begin{remark}\label{r5}
Observe that hypothesis $(V_1)_0$ is sufficient to guarantee equivalence of the concerned norms, in view of (\ref{e00})-(\ref{e000}) and since $0 \notin \sigma(A_0)$. In fact, it could be weakened asking for $f'(0)\notin \sigma_{ess}(A)$, namely, $0\notin \sigma_{ess}(A_0)$ instead of $0 \notin \sigma(A_0)$, for details see \cite{MS}. However, since $V$ is periodic, we have $\sigma(A) = \sigma_{ess}(A)$ and these hypotheses are the same.
\end{remark}
\qquad Following the variational approach, we associate to (\ref{prob_0}) the functional $I:E\to \mathbb{R}$ given by
\begin{equation}\label{I}
I(u) = \dfrac{1}{2}\Big(||u^+||^2 -||u^-||^2 \Big) - \int_{\mathbb{R}^N}F_0(u(x))\,dx,
\end{equation}
which is strongly indefinite and belongs to $C^1(E,\mathbb{R})$ in view of the previous hypotheses. Our goal is to apply an abstract linking result due to G. Li and A. Szulkin developed in \cite{LS}, which is going to provide us a Cerami sequence for functional $I$ on a positive level $c \in \mathbb{R}$. 
We recall that $(u_n)\subset E$ is a Cerami sequence for $I$ if
\[
\displaystyle\sup_n|I(u_n)| <+\infty \quad \text{and} \quad ||I'(u_n)||_{E'}(1 + ||u_n||)\to 0 \quad \text{as} \quad n\to+\infty.
\]
Furthermore, $(u_n)$ is called a $(C)_c$ sequence, or a Cerami sequence on the level $c$, if besides that it satisfies $I(u_n)\to c$ as $n\to +\infty$. 
The abstract linking theorem mentioned above is stated now:
\begin{theorem}[Theorem 2.1 \cite{LS}]\label{AR}
Let $E = E^+ \oplus E^-$ be a separable Hilbert space with $E^-$ orthogonal to $E^+$. Suppose
\medskip
\newline	
(i) $I(u) = \dfrac{1}{2}\Big(||u^+||^2 -||u^-||^2\Big) - \varphi(u)$, where $\varphi \in C^1(E,\mathbb{R})$ is bounded
below, weakly sequentially lower semi-continuous and $\varphi'$ is weakly sequentially continuous;
\medskip
\newline
(ii) There exist $z_0 \in E^+ \slash \{0\}$, $\beta>0$ and $R>r>0$ such that $I|_N\geq \beta$ and  $I|_{\partial M}\leq 0$.
\medskip
\newline
Then there exists a $(C)_c$-sequence for $I$, where 
\[
M:=\{u= u^-+tz_0: u^- \in E^-, \; ||u||\leq R, \; t\geq0\}, \; N:=\partial B_r\cap E^+,
\]
\[
\Gamma:=\left\{
\begin{array}{ll}
h \in C(M\times[0,1],E):h \, is \, admissible, \; h(u,0)=u\\
 and \;
I(h(u,s))\leq\max\{I(u),-1\}, \; \forall \; s \in[0,1]
\end{array}
\right\}
\]
and
\[
c := \inf_{h \in \Gamma}\sup_{u \in M}I(h(u,1)).
\]
\end{theorem}

\qquad Next section is dedicated to prove that under our hypotheses, functional $I$ satisfies $(i)-(ii)$ in Theorem \ref{AR}, providing the desired $(C)_c$ sequence. Posteriorly, having this sequence in hand, in virtue of the assumptions on $f$ we are going to be able to show it is bounded. Finally, by means of an indirect argument, last section of this paper sets up a nontrivial critical point to $I$, which is the weak limit of a suitable Cerami subsequence. Therefore, we obtain a nontrivial weak solution to (\ref{prob_0}).

\section{Establishing the Linking Structure}

\medskip
\qquad Henceforth, we aim to check functional $I$ defined in (\ref{I}) satisfies assumptions $(i)-(ii)$ in Theorem \ref{AR}. Next lemma verifies $(i)$.

\begin{lemma}\label{l1}
Under assumptions $(f_1)_0-(f_2)_0$ on $f_0$ and $(V_1)_0$ on $V_0$, functional $I$ satisfies $(i)$ in Theorem \ref{AR}. 	
\end{lemma}
\begin{proof}
In view of $(f_1)_0$ one has $F_0\geq0$ and defining
\[
\varphi(u):= \int_{\mathbb{R}^N}F_0(u(x))\;dx, \; \text{for \ all} \; u \in E, 
\]
it follows that $\varphi\geq 0$, hence it is bounded from below. In addition, if $(u_n)\in E$ and ${u_n \rightharpoonup u}$ in $E$, then $u_n\to u$ in $L^2_{loc}(\mathbb{R}^N)$ and $u_n(x)\to u(x)$ almost everywhere in $\mathbb{R}^N$. Since $F_0$ is continuous, one has $F_0(u_n(x))\to F_0(u(x))$ almost everywhere in $\mathbb{R}^N$. Hence, applying Fatou's Lemma, we arrive at
\[
\varphi(u) = \int_{\mathbb{R}^N}\liminf_{n \to +\infty}F_0(u_n(x))\;dx\leq \liminf_{n \to +\infty}\int_{\mathbb{R}^N}F_0(u_n(x))\;dx = \liminf_{n \to +\infty}\varphi(u_n),
\]
namely, $\varphi$ is weakly sequentially lower semi-continuous. Furthermore, for all $u, v \in E$ one has
\[
\varphi'(u)v = \int_{\mathbb{R}^N}f_0(u(x))v(x)\;dx,
\]
hence if $u_n \rightharpoonup u$ in $E$, it implies that $\varphi'(u_n)v \to \varphi'(u)v$ for all $v \in E$, namely $\varphi'$ is weakly sequentially continuous. Indeed, since $(u_n)$ is bounded in $E$, from $(f_1)_0-(f_2)_0$ it is standard that $(f_0(u_n))$ is bounded in $L^2(\mathbb{R}^N)$. In addition, since $u_n(x)\to u(x)$ almost everywhere in $\mathbb{R}^N$, then $f_0(u_n(x))\to f_0(u(x))$ almost everywhere in $\mathbb{R}^N$ and thus, $f_0(u_n)\rightharpoonup f_0(u)$ in $L^2(\mathbb{R}^N)$, hence
\[
\varphi'(u_n)v = \int_{\mathbb{R}^N}f_0(u_n(x))v(x)\;dx \to \int_{\mathbb{R}^N}f_0(u(x))v(x)\;dx = \varphi'(u)v \quad \text{for \ all} \; v \in E,
\]
as $n \to +\infty.$ Therefore, in view of (\ref{I}), $I$ satisfies $(i)$ in Theorem \ref{AR}.
\end{proof}
\medskip

\qquad 
%Note that it does not matter if $V_0$ is periodic, even asymptotically periodic, proof of Lemma \ref{l1} is exactly the same.
From hypothesis $(V_1)_0$ one has 
%\[
${a_0> \sigma^+ - f'(0)>0},$
%\]
hence the spectral family of operator $A_0$ ensures the existence of some $z_0 \in E^+$ with $||z_0|| =1$ and satisfying
\begin{equation}\label{le}
\big(\sigma^+ - f'(0)\big)||z_0||_{L^2(\mathbb{R}^N)}^2\leq ||z_0||^2 < a_0||z_0||^2_{L^2(\mathbb{R}^N)}.
\end{equation}
Choosing such $z_0$ and defining $M$ and $N$ as in Theorem \ref{AR}, condition $(ii)$ is given by the linking geometry proved as follows.

\begin{lemma} \label{l2}
Assuming $(f_1)_0-(f_2)_0$ and $(V_1)_0$, there exist $\beta >0$ and $R>r>0$ such that $I|_N \geq \beta$ and $I|_{\partial M}\leq 0$.
\end{lemma}

\begin{proof}
From $(f_1)_0-(f_2)_0$, given $\varepsilon>0$, there exist $C_\varepsilon>0$ such that
\[
|F_0(t)|\leq \dfrac{1}{2}\varepsilon|t|^2 + \dfrac{C_\varepsilon}{p}|t|^p \quad \text{for \ all} \; t \in [0,1] \quad  \text{and \ for \ some} \; p \in [2,2^*).
\]
Thus, given $u^+ \in N$ one has
\begin{eqnarray*}
I(u^+) &=& \dfrac{1}{2}||u^+||^2 - \int_{\mathbb{R}^N}F_0(u^+(x))\;dx\\
&\geq&\dfrac{1}{2}||u^+||^2 - \dfrac{1}{2}\varepsilon||u^+||^2_{L^2(\mathbb{R}^N)} - \dfrac{C_\varepsilon}{p}||u^+||^p_{L^p(\mathbb{R}^N)}\\
&\geq& \dfrac{r^2}{2}\Big(1 - \varepsilon C^2_2 - \dfrac{C_\varepsilon}{p}C_p^pr^{p-2}\Big)\\
&=&\beta>0,
\end{eqnarray*}
where, $C_q>0$ is the constant given by the continuous embedding $E\hookrightarrow L^q(\mathbb{R}^N)$ for $q \in [2,2^*]$ and we choose $\varepsilon>0, \; r>0$ small enough to guarantee $\beta>0$.

\qquad On the other hand, suppose by contradiction that $I|_{\partial M}\leq0$ does not hold true for all $R>r$. Then, given $R>r$ there exist $u^-_R \in E^-$ and $t_R>0$ such that $u_R = u^-_R + t_Rz_0,$ \ $||u_R|| = R$ and $I(u_R)>0$. Choosing a sequence $R_n \to +\infty$ as $n \to +\infty$ one obtain a sequence $(u_n) \subset E$ with $u_n = u^-_n + t_nz_0, \; ||u_n||= R_n \to +\infty$ as $n \to +\infty$ and $I(u_n)>0$ for all $n\in \mathbb{N}$.

\qquad Defining $v_n:= \dfrac{u_n}{||u_n||} = v^-_n + s_nz_0$ with $v^-_n = \dfrac{u^-_n}{||u_n||}$ and $s_n = \dfrac{t_n}{||u_n||}$, it follows that $(v_n)\subset E$ is bounded with
\begin{equation}\label{lge00}
1 = ||v_n||^2 = ||v^-_n||^2 + s_n^2,
\end{equation}
hence $v_n \rightharpoonup v$ in $E$, $v^-_n \rightharpoonup v^-$ in $E$ and $s_n \to s$ in $\mathbb{R}^+$, up to subsequences, where 
\begin{equation}\label{lge000}
v = v^- + sz_0. 
\end{equation}
Moreover, by assumption
\begin{eqnarray}\label{lge1}
0< \dfrac{I(u_n)}{||u_n||^2} &=& \dfrac{1}{2}\Big(s_n^2 - ||v_n^-||^2\Big) - \int_{\mathbb{R}^N}\dfrac{F_0(u_n(x))}{||u_n||^2}\;dx\nonumber\\
&=&s_n^2 -\dfrac{1}{2} - \int_{\mathbb{R}^N}\dfrac{F_0(u_n(x))}{||u_n||^2}\;dx\nonumber\\
&\leq& s_n^2 - \dfrac{1}{2},
\end{eqnarray}
which, in view of (\ref{lge00}), implies that $\dfrac{1}{2}\leq s^2\leq1$. 

\qquad From (\ref{le}) it is possible to choose $\Omega \subset \mathbb{R}^N$ a bounded domain such that
\begin{equation}\label{lge0}
1 = ||z_0||^2 < a_0\int_{\Omega}|z_0(x)|^2\;dx.
\end{equation} 
Then  (\ref{lge00}) and (\ref{lge0}) imply that
\begin{equation}\label{lge2}
s^2 - s^2a_0\int_{\Omega}|z_0(x)|^2\;dx -(1-s^2) - a_0\int_{\Omega}|v^-(x)|^2\;dx<0.
\end{equation}
On the other hand, from (\ref{lge1}) we get
\begin{equation}\label{lge3}
0< s^2_n -\dfrac{1}{2} - \int_{\Omega}\dfrac{F_0(u_n(x))}{||u_n||^2}\;dx.
\end{equation}
In virtue of (\ref{lge000}) and since $s^2\geq \dfrac{1}{2}$, one has $v \not=0$ and from (\ref{lge0}) it yields $|\Omega \cap supp(v)|>0$. Provided that $|u_n(x)|\to +\infty$ for all $x \in \Omega\cap supp(v)$, from $(f_2)_0$ it follows that
 \[
 \lim_{n\to+\infty}\dfrac{F_0(u_n(x))}{||u_n||^2} = \lim_{n\to+\infty} \dfrac{F_0(u_n(x))}{u^2_n(x)}v^2_n(x) \geq \dfrac{a_0}{2}v^2(x),
 \]
 almost everywhere in $\Omega \cap supp(v)$. Furthermore, $v_n \to v$ in $L^2(\Omega)$, up to subsequences, as $n \to +\infty$, then
 by Fatou's Lemma, it follows that
 \[
 \liminf_{n\to+\infty}\int_{\Omega}\dfrac{F_0(u_n(x))}{||u_n(x)||^2}\;dx \geq
 \dfrac{a_0}{2}\int_{\Omega}|v(x)|^2\;dx,
 \]
 as $n\to +\infty.$ Since (\ref{lge000}) implies that $v(x) = v^-(x) + sz_0(x)$, passing (\ref{lge3}) to the limit and multiplying by two, we arrive at
 \begin{eqnarray*}
  0&\leq& 2s^2 - {1} - {a_0}\int_{\Omega}\Big(s^2|z_0(x)|^2 + |v^-(x)|^2\Big)\;dx,\\
  &=&s^2 - s^2{a_0}\int_{\Omega}|z_0(x)|^2\;dx - (1-s^2) - {a_0}\int_{\Omega}|v^-(x)|^2\;dx,
 \end{eqnarray*}
   which contradicts (\ref{lge2}). Therefore, we conclude there exists $R>0$ large enough satisfying the required inequality.
\end{proof}

\qquad Note that the strict inequality in (\ref{le}), inherited from $(V_1)_0$, was essential to obtain the suitable $z_0$ to prove Lemma \ref{le} and obtain $(ii)$.
%In addition, if $V$ is asymptotically periodic and satisfies $(V_1)_0$, then Lemma \ref{l2} is proved without changes.
Furthermore, up to now we have verified all assumptions in Theorem \ref{AR}, we obtain a Cerami sequence for $I$ on the level $c\geq\beta>0$, under the hypotheses of Theorem \ref{main}.

\section{Boundedness of Cerami Sequences}

\qquad In this section we are going to prove that, under the assumptions of Theorems \ref{main}, every Cerami sequence for $I$ is bounded, particularly, those found by Theorem \ref{AR} in the previous section. Hereafter, we prove two lemmas treating singly both cases, asymptotically linear and superlinear. Hence, firstly it is necessary identify hypotheses $(f_2)'_0$ and $(f_2)_0''$ corresponding to $(f_2)'$ and $(f_2)''$, respectively, in order to prove these lemmas. They are stated below:
\bigskip
\newline
$(f_2)_0'$\quad For $a_0:= a - f'(0)$ it holds
\[
\displaystyle\lim_{|s| \to +\infty}\dfrac{f_0(s)}{s}= a_0;
\]
\newline
$(f_2)_0''$\quad For some $p\in(2,2^*)$ it holds
\[
\displaystyle\lim_{|s| \to +\infty}\dfrac{f_0(s)}{s}= +\infty\quad \text{and} \quad \limsup_{|s|\to+\infty}\dfrac{f_0(s)}{s^{p-1}} <+\infty.
\]
\medskip
%As in Lemma \ref{l2}, for the case where $V$ is asymptotically periodic next lemma is proved exactly as follows.

\qquad With hypothesis $(f_2)'_0$ in hand, the asymptotically linear case is treated as follows.
\medskip
\begin{lemma}\label{B}
Let $(u_n)$ be a $(C)_c$ sequence for $I$. Under the assumption of $(V_1)_0-(V_2)_0$, ${(f_1)_0-(f_2)'_0}$ and $(f_3)$, it follows that $(u_n)$ is bounded.
\end{lemma}	

\begin{proof}
Arguing by contradiction we suppose $||u_n||\to +\infty$ as $n \to +\infty$ up to subsequences. Defining $v_n := \dfrac{u_n}{||u_n||}$, we have $(v_n)\subset E$ bounded, hence $v_n \rightharpoonup v$ in $E$ and $v_n \to v$ in $L^q_{loc}(\mathbb{R}^N)$ for $q \in [2,2^*)$, then $v_n(x) \to v(x)$ almost everywhere in $\mathbb{R}^N$.

\qquad Suppose that $v\not=0$, hence there exists $\Omega \subset \mathbb{R}^N$ such that $v(x)\not=0$ for all $x \in \Omega$ and $|\Omega|>0$. Since $u_n(x) = ||u_n||v_n(x) \to +\infty$ for all $x \in \Omega$, then  $(f_3)$ implies that
\begin{eqnarray}\label{lbe1}
\liminf_{n \to +\infty}\int_{\mathbb{R}^N}Q_0(u_n(x))\;dx &\geq&\liminf_{n \to +\infty}\int_{\Omega}Q_0(u_n(x))\;dx\nonumber\\
&\geq& \int_{\Omega}\liminf_{n \to +\infty}Q_0(u_n(x))\;dx\nonumber\\
&=& +\infty
\end{eqnarray}

\qquad On the other hand, since $(u_n)$ is a $(C)_c$ sequence for $I$, we have
\begin{eqnarray}\label{lbe0}
2c + o_n(1) &=& 2I(u_n) - I'(u_n)u_n \nonumber\\
&=& \int_{\mathbb{R}^N}Q_0(u_n(x))\;dx,
\end{eqnarray}
contradicting (\ref{lbe1}). Thus, $v = 0$, the null function in $E$.

\qquad Furthermore, writing $u_n = u_n^+ + u_n^-$, since $(u_n)$ is a Cerami sequence it follows that
\begin{eqnarray}\label{lbe2}
o_n(1) &=& I'(u_n)\dfrac{\big(v_n^+ - v_n^-\big)}{||u_n||}\nonumber\\
&=& \Big(v_n,v_n^+-v_n^-\Big) - \int_{\mathbb{R}^N}\dfrac{f_0(u_n(x))}{||u_n||}\big(v_n^+(x) - v_n^-(x)\big)\;dx\nonumber\\
&=&||v_n^+||^2 + ||v_n^-||^2 - \int_{\mathbb{R}^N}\dfrac{f_0(u_n(x))}{u_n(x)}\big[(v_n^+(x))^2 - (v_n^-(x))^2\big]\;dx\nonumber\\
&=& 1 - \int_{\mathbb{R}^N}\dfrac{f_0(u_n(x))}{u_n(x)}\big[(v_n^+(x))^2 - (v_n^-(x))^2\big]\;dx.
\end{eqnarray}
Hence, 
\begin{equation}\label{lbe4}
\displaystyle\int_{\mathbb{R}^N}\dfrac{f_0(u_n(x))}{u_n(x)}\big[(v_n^+(x))^2 - (v_n^-(x))^2\big]\;dx \to 1 \quad \text{as} \quad n \to +\infty.
\end{equation}

\qquad Since \; $E \hookrightarrow L^2(\mathbb{R}^N)$ continuously, \; let $C_2>0$ be \; the constant \; such that ${||u||_{L^2(\mathbb{R}^N)}\leq C_2||u||}$ for all $u \in E$. Taking $0<\varepsilon < \dfrac{1}{C_2^2}$, in view of $(f_1)_0$ there exists $\delta>0$ such that 
\[
\left|\dfrac{f_0(s)}{s}\right|< \varepsilon \quad \text{for} \quad 0< |s| < \delta.
\]
Now, for all $n \in \mathbb{N}$ we define $\Omega_n := \{x\in \mathbb{R}^N: |u_n(x)| < \delta \}$ and then we obtain
\begin{eqnarray}\label{lbe3}
\displaystyle\int_{\Omega_n}\left|\dfrac{f_0(u_n(x))}{u_n(x)}\right|\Big[(v_n^+(x))^2 + (v_n^-(x))^2\Big]\;dx &\leq&\varepsilon\int_{\Omega_n}\Big[(v_n^+(x))^2 + (v_n^-(x))^2\Big]\;dx \nonumber\\
&\leq& \varepsilon\big[||v_n^+||_{L^2(\mathbb{R}^N)}^2 + ||v_n^-||^2_{L^2(\mathbb{R}^N)}\big]\nonumber\\
&\leq& {\varepsilon}C_2^2||v_n||^2\nonumber\\
&=& \varepsilon C_2^2,
\end{eqnarray}
which implies that
\[
\liminf_{n \to +\infty}\displaystyle\int_{\Omega_n}\dfrac{f_0(u_n(x))}{u_n(x)}\Big[(v_n^+(x))^2 - (v_n^-(x))^2\Big]\;dx \leq \varepsilon C_2^2 <1.
\]
Thus, from (\ref{lbe4}) we conclude that
\begin{equation}\label{lbe5}
\liminf_{n \to +\infty}\displaystyle\int_{\mathbb{R}^N \slash\Omega_n}\dfrac{f_0(u_n(x))}{u_n(x)}\Big[(v_n^+(x))^2 - (v_n^-(x))^2\Big]\;dx>0.
\end{equation}

\qquad On the other hand, from $(f_1)_0-(f_2)'_0$, one has $\left|{f_0(s)}/{s}\right|$ bounded, hence applying H\"older inequality for $q \in (2,2^*)$ we arrive at
\begin{equation}\label{lbe6}
\int_{\mathbb{R}^N \slash\Omega_n}\dfrac{f_0(u_n(x))}{u_n(x)}\Big[(v_n^+(x))^2 - (v_n^-(x))^2\Big]\;dx \leq C |\mathbb{R}^N\slash \Omega_n|^{\frac{q-2}{q}}||v_n||^{\frac{2}{q}}_{L^q(\mathbb{R}^N)}.
\end{equation}

\qquad Now, we observe that if $(v_n)$ is a non-vanishing sequence, then there exist $r, \; \eta >0$ and a sequence $(y_n)\in \mathbb{R}^N$ such that
\begin{equation}\label{lbe7}
\limsup_{n\to +\infty}\int_{B_r(y_n)}|v_n(x)|^2\;dx > \eta.
\end{equation}
Defining $\tilde{v}_n(x):= v_n(x +y_n)$ it implies that $(\tilde{v}_n)$ is also bounded and hence $\tilde{v}_n \rightharpoonup \tilde{v}$ in $E$ and $\tilde{v}_n \to \tilde{v}$ in $L^2_{loc}(\mathbb{R}^N)$, then
\begin{equation}\label{lbe8}
||\tilde{v}||^2_{L^2(B_r(0))} = \lim_{n \to +\infty}\int_{B_r(0)}|\tilde{v}_n(x)|^2\;dx = \limsup_{n\to +\infty}\int_{B_r(y_n)}|v_n(x)|^2\;dx > \eta.
\end{equation}
Thus, $\tilde{v}\not=0$, namely, there exist $\tilde{\Omega}\subset \mathbb{R}^N$ such that $|\tilde{\Omega}|>0$
and $\tilde{v}(x)\not=0$ for all $x \in \tilde{\Omega}$. Defining also 
\begin{equation}\label{lbe08}
\tilde{u}_n(x) := {u}_n(x+y_n),
\end{equation}
in view of the equivalence between $E$-norm and $H^1(\mathbb{R}^N)$-norm we have $||\tilde{u}_n|| \to +\infty$ as $n\to +\infty$ and since
\[
\int_{\mathbb{R}^N}Q_0(u_n(x))\;dx = \int_{\mathbb{R}^N}Q_0(\tilde{u}_n(x))\;dx,
\]
arguing as in (\ref{lbe1})-(\ref{lbe0}) with $(\tilde{v}_n)$ instead of $({v}_n)$ we arrive at contradiction. Therefore, $(v_n)$ cannot be a non-vanishing sequence, namely, $(v_n)$ is a vanishing sequence and hence
\begin{equation}\label{lbe09}
\limsup_{n\to +\infty}\sup_{y \in \mathbb{R}^N}\int_{B_r(y)}|v_n(x)|^2\; dx = 0, \quad \forall \; r>0.
\end{equation}
Applying Lion's Lemma, since $(v_n)$ is a vanishing sequence it yields $||v_n||_{L^q(\mathbb{R}^N)} \to 0$ as $n \to +\infty$ for $q \in (2,2^*)$ and in view of $(\ref{lbe5})-(\ref{lbe6})$ it follows that
\begin{equation}\label{lbe9}
{|\mathbb{R}^N\slash\Omega_n| \to +\infty}  \quad \text{as}  \quad n \to +\infty.
\end{equation}

\qquad On the other hand, without loss of generality, from $(f_3)$ there exists $M>\delta$ such that $Q_0(s)>1$ for all $s>M$. Then, we define $U_n:=\{x\in \mathbb{R}^N: |u_n(x)|>M\}$ and provided that $(u_n)$ is a $(C)_c$ sequence for $I$, it follows that
\[
2c + o_n(1) = 2I(u_n) - I'(u_n)u_n \geq \int_{U_n}Q_0(u_n(x))\;dx \geq |U_n|,
\]
hence, $(|U_n|)$ is a bounded sequence. Defining also,
\[
W_n:= \{x\in \mathbb{R}^N: \delta\leq|u_n(x)|\leq M\},
\]
we obtain $\mathbb{R}^N\slash \Omega_n = U_n \cup W_n$, with $U_n$ and $W_n$ disjointed sets, then $|\mathbb{R}^N\slash \Omega_n|= |U_n| + |W_n|$ and since $(|U_n|)$ is bounded, we conclude that $|W_n|\to +\infty.$

\qquad Moreover, provided that $Q_0(s)>0$ is a continuous function, it implies that 
\[
{m:= \displaystyle\inf_{s\in[\delta,M]}Q_0(s)}
\]
is a positive constant. Thus,
\begin{equation}\label{lbe10}
\int_{\mathbb{R}^N}Q_0(u_n(x))\;dx \geq \int_{W_n}Q_0(u_n(x))\;dx \geq m|W_n| \to +\infty, 
\end{equation}
as $n \to +\infty$. Nevertheless, since $(u_n)$ is a $(C)_c$ sequence (\ref{lbe10}) also contradicts (\ref{lbe0}). Therefore, $(u_n)$ is bounded.
\end{proof}	

\qquad Now, considering the superlinear case, next lemma gives the boundedness.
\begin{lemma}\label{B2}
Let $(u_n)$ be a $(C)_c$ sequence for $I$. Under the assumption of $(V_1)_0-(V_2)_0$, ${(f_1)_0-(f_2)''_0}$ and $(f_4)$, it follows that $(u_n)$ is bounded.
\end{lemma}
\begin{proof}
	Let $b>0$  be  a  constant  such  that  $c-b>0$, since $(u_n)$ is a $(C)_c$ sequence for $I$ it follows that ${(u_n) \subset I^{-1}([c-b,c+b])}$ for sufficiently large $n$. Moreover, $(u_n)$ satisfies
	\begin{eqnarray}\label{b2e1}
	o_n(1)= I'(u_n){u}^+_n =||{u}^+_{n}||^2 - \int_{\mathbb{R}^N}f_0(u_n(x)){{u}^+_n(x)}\; dx
	\end{eqnarray}
	and
	\begin{eqnarray}\label{b2e2}
	o_n(1)= {I'(u_n)({u}^-_n)} =-||{u}^-_{n}||^2 - \int_{\mathbb{R}^N}f_0(u_n(x)){{u}^-_n(x)} \; dx. 
	\end{eqnarray}
	Subtracting (\ref{b2e2}) from (\ref{b2e1}), and using $(f_1)_0-(f_2)''_0$, it yields
	\begin{eqnarray}\label{b2e3}
	||u_n||^2&=& o_n(1) + \int_{\mathbb{R}^N}f_0(u_n)\Big(u^+_n(x) -u^-_n(x)\Big)\, dx\nonumber\\
	&\leq& C + 2\int_{ \mathbb{R}^N}f_0(u_n)|u_n(x)|\ dx\nonumber\\
	&\leq& C + 2\int_{\mathbb{R}^N}\Big(\varepsilon|u_n(x)|^2+C_\varepsilon|u_n(x)|^p\Big)\, dx\nonumber\\
	&\leq& C + 2C^2_2\varepsilon||u_n||^2 + 2C_\varepsilon||u_n||^p_{L^p(\mathbb{R}^N)},
	\end{eqnarray}
	for arbitrary $\varepsilon>0$, some constants $C, \ C_\varepsilon >0$ and for $C_2>0$ a constant given by the embedding $E\hookrightarrow L^2_(\mathbb{R}^N)$. Hence, in view of (\ref{b2e3}), one has
	\begin{equation}\label{b2e4}
	\big(1 - 2C_2^2\varepsilon\big)||u_n||^2 \leq C + 2C_\varepsilon||u_n||^p_{L^p(\mathbb{R}^N)},
	\end{equation}
	which means that for sufficiently small $\varepsilon>0$, the boundedness of $(u_n)$ in $E$ is directly related to the boundedness of $(u_n)$ in $L^p(\mathbb{R}^N)$. So, it turns out to be necessary to estimate the  $||u_n||_{L^p(\mathbb{R}^N)}$. Observe that from $(f_4)$ and since $(u_n)$ is a Cerami sequence, for some constant $M>0$, one has
	\begin{eqnarray}\label{b2e5}
	M&\geq& I(u_n) - \dfrac{1}{2}I'(u_n)(u_n)\nonumber\\
	&=& \dfrac{1}{2}\int_{\mathbb{R}^N}f_0(u_n)u_n(x)\; dx	- \int_{\mathbb{R}^N}F_0(u_n(x))\;dx\nonumber\\
	&=&\dfrac{1}{2}\int_{\mathbb{R}^N}Q_0(u_n(x)) \ dx\nonumber\\
	&\geq& \dfrac{c_0}{2}\int_{\mathbb{R}^N}|u_n(x)|^\theta \;dx\nonumber\\
	&=&\dfrac{c_0}{2}||u_n||^\theta_{L^\theta(\mathbb{R}^N)},
	\end{eqnarray}
	thus, $(u_n)$ is a bounded sequence in $L^\theta(\mathbb{R}^N)$.
	Hence, if $p=\theta$ from (\ref{b2e4}) it implies that $(u_n)$ bounded in $E$. If not, first we consider the case $2<p<\theta$, then there exists $t \in (0,1)$ such that $p = 2t+ (1-t)\theta$, hence applying H\"older Inequality and in view of (\ref{b2e5}) it follows that
	\begin{eqnarray}\label{b2e6}
	||u_n||^p_{L^p(\mathbb{R}^N)}&\leq& ||u_n||^{2t}_{L^2(\mathbb{R}^N)}||u_n||^{(1-t)\theta}_{L^\theta(\mathbb{R}^N)}\nonumber\\
	&\leq&\left(\dfrac{2M}{c_0}\right)^{1-t}C_2^{2t}||u_n||^{2t}.
	\end{eqnarray}
	Substituting (\ref{b2e6}) in (\ref{b2e4}), it yields
	\begin{equation}\label{b2e7}
	\big(1 - 2C_2^2\varepsilon\big)||u_n||^2 \leq C + 2C_\varepsilon\left(\dfrac{2M}{c_0}\right)^{1-t}C_2^{2t}||u_n||^{2t},
	\end{equation}
	and since $2t<2$, choosing $\varepsilon>0$ small enough, it ensures that $(u_n)$ is bounded in $E$. Now, it remains to consider the case $\theta<p<2^*$, then there exists $t\in(0,1)$ such that $p = 2^*t + (1-t)\theta$. Again from H\"older Inequality and using the boundedness in (\ref{b2e5}) it follows that
	
	\begin{eqnarray}\label{b2e8}
	||u_n||^p_{L^p(\mathbb{R}^N)}&\leq& ||u_n||^{2^*t}_{L^{2^*}(\mathbb{R}^N)}||u_n||^{(1-t)\theta}_{L^\theta(\mathbb{R}^N)}\nonumber\\
	&\leq&\left(\dfrac{2M}{c_0}\right)^{1-t}C_{2^*}^{2^*t}||u_n||^{2^*t},
	\end{eqnarray}
	where $C_{2^*}>0$ is the constant given by the embedding $E\hookrightarrow L^{2^*}(\mathbb{R}^N)$. 
	Substituting (\ref{b2e8}) in (\ref{b2e4}), it yields
	\begin{equation}\label{b2e9}
	\big(1 - 2C_2^2\varepsilon\big)||u_n||^2 \leq C + 2C_\varepsilon\left(\dfrac{2M}{c_0}\right)^{1-t}C_{2^*}^{2^*t}||u_n||^{2^*t},
	\end{equation}
	and provided that $\theta >\dfrac{N}{2}(p-2)$, one has $2^*t <2$, hence choosing $\varepsilon>0$ small enough, it ensures that $(u_n)$ is also bounded in $E$ for this case. Therefore, the result holds.
	\end{proof}

\section{Proof of Main Results}

\qquad Previous sections have proved the existence of $(u_n)$ a bounded $(C)_c$ sequence for $I$ in both cases, asymptotically linear and superlinear. Now, we are able to prove our main result. 

\begin{proof}[Proof of Theorem \ref{main}]
Knowing that $(u_n)$ is a bounded $(C)_c$  sequence for $I$, it implies that $u_n \rightharpoonup u$ in $E$ and $u_n \to u$ in $L^2_{loc}(\mathbb{R}^N)$. Given $\varphi \in C^\infty_0(\mathbb{R}^N)$, let $K\subset\mathbb{R}^N$ be the compact support of $\varphi$. Since $(u_n)$ is a $(C)_c$ sequence for $I$, from weak convergence and  Lebesgue Dominated Convergence Theorem, it follows that
\begin{eqnarray}\label{e0}
o_n(1) &=& I'(u_n)\varphi\nonumber\\
 &=& \Big(u_n^+ - u_n^-,\varphi\Big) - \int_{\mathbb{R}^N}f_0(u_n(x))\varphi(x)\;dx\nonumber\\ 
 &=& \Big(u^+ - u^-,\varphi\Big) - \int_{K}f_0(u(x))\varphi(x)\;dx + o_n(1)\nonumber\\
 &=& I'(u)\varphi + o_n(1).
\end{eqnarray}
Thus, $I'(u)\varphi =0$, and since $\varphi$ is arbitrary, by density we obtain $I'(u)\equiv0$ and therefore $u$ is a critical point of $I$. 

\qquad If $u\not=0$, we obtain a nontrivial critical point of $I$, if not, we observe that $u_n \not\to 0$ in $E$, provided that $I(u_n)\to c \not=0$. Furthermore, since $c>0$, we claim $(u_n)$ cannot be a vanishing sequence. In fact, given $\varepsilon>0$ from $(f_1)_0-(f_2)_0$ there exists $C_\varepsilon>0$ such that
\[
\int_{\mathbb{R}^N}|f_0(u(x))u(x)|\;dx \leq \varepsilon||u||^2_{L^2(\mathbb{R}^N)} + C_\varepsilon||u||^p_{L^p(\mathbb{R}^N)},
\]
thus, if $(u_n)$ vanishes, since $\varepsilon$ is arbitrary, we get that 
\begin{equation}\label{nce5}
\displaystyle\int_{\mathbb{R}^N}f_0(u_n(x))\big|u_n^+(x)-u_n^-(x)\big|\;dx \to 0, \quad \text{as} \quad n\to +\infty.
\end{equation}
Since $(u_n)$ is a $(C)_c$ sequence for $I$ and in virtue of (\ref{nce5}) we arrive at
\begin{equation}\label{nce6}
o_n(1) = I'(u_n)\big(u_n^+ - u_n^-\big)  + \int_{\mathbb{R}^N}f_0(u_n(x))\big(u_n^+(x) - u_n^-(x)\big)\;dx = ||u_n||^2,
\end{equation} 
contradicting that $\displaystyle\liminf_{n \to +\infty}||u_n||^2\geq \displaystyle\liminf_{n \to +\infty}||u_n^+||^2\geq \displaystyle\lim_{n \to +\infty}2I(u_n) = c>0$. Therefore $(u_n)$ is a non-vanishing sequence.

\qquad From the equivalence of norms, the boundedness of $(u_n)$ implies the boundedness of $(\tilde{u}_n)$ in (\ref{lbe08}), then we assume $\tilde{u}_n \rightharpoonup \tilde{u}$ in $E$, up to subsequences, and since $(u_n)$ is non-vanishing, it follows that (\ref{lbe7})-(\ref{lbe8}) hold true, with $u_n, \; \tilde{u}_n$ instead of $v_n,\; \tilde{v}_n$ respectively, thus, $\tilde{u}\not= 0$. On the other hand, in view of $(V_1)_0$ we have that $I$ is $ \mathbb{Z}^N$-invariant, up to a translation, hence $I(u_n) = I(\tilde{u}_n)$ {and} $I'(u_n)\equiv I'(\tilde{u}_n),$ therefore $(\tilde{u}_n)$ is also a $(C)_c$ sequence for $I$. Arguing as in (\ref{e0}), but with $(\tilde{u}_n)$ instead of $(u_n)$, we conclude that $\tilde{u}$ is a nontrivial critical point of $I$. Since $I \in C^1(E,\mathbb{R})$ such a critical point is a nontrivial weak solution to problem $(\ref{prob_0})$. Therefore, $\tilde{u}$ is equivalently a nontrivial weak solution to (\ref{prob}), as desired.
\end{proof}

\end{document}